\documentclass[12pt]{article} 
\usepackage{subfig}
\usepackage{tikz}
\usepackage{graphicx}
\usepackage{mathrsfs}
\usetikzlibrary{positioning}
\usepackage{epstopdf} %converting to PDF
\usepackage{authblk}
\newcommand{\be}{\begin{equation}}
\newcommand{\ee}{\end{equation}}
\newcommand{\ba}{\begin{array}}
\newcommand{\ea}{\end{array}}

\newcommand{\bea}{\begin{eqnarray*}}
\newcommand{\eea}{\end{eqnarray*}}
\newcommand{\bean}{\begin{eqnarray}}
\newcommand{\eean}{\end{eqnarray}}

\newtheorem{lemma}{Lemma}[section]

\newtheorem{definition}{Definition}[section]
\newtheorem{theorem}{Theorem}[section]

\newtheorem{proposition}{Proposition}[section]
\usepackage{amsmath, amssymb}
\usepackage{latexsym}

\def\Box{\leavevmode\vbox{\hrule
     \hbox{\vrule\kern5pt\vbox{\kern5pt}%
           \vrule}\hrule}}

\numberwithin{equation}{section}
\usepackage{float}
\usepackage{caption}
 \DeclareCaptionFormat{sanslabel}{#3}%
\newcommand*{\email}[1]{{E-mail:  #1 }}
\author[1]{Meniar {$Haddad$}\thanks{\email{\texttt{Corresponding author: Meniar.Haddad@fst.rnu.tn}}}}
\author[2]{Wafa {Djobbi}\thanks{\email{\texttt{wafadjobbi100@gmail.tn}}}}
\title{On The Babenko-Bechner-Type Inequality   associated  with the Weinstein  Operator}
\begin{document}
%\author[1]{Meniar Haddad}
%\author[2]{Wafa Djobbi}
\affil[1]{ University of carthage, Faculty of sciences of Bizerte, 7021 Zarzouna, Departement of Mathematics. University of Tunis El Manar, Faculty of sciences of Tunis, Research Laboratories of Mathematics Analysis and Applications LR11ES11 Tunisia.}
\maketitle
\begin{abstract}
In this paper, we study the Babenko-Bechner-type inequality for the Fourier Weinstein transform $\mathcal{F} _{w}^{ \alpha ,d}$ associated with the Weinstein operator $ \Delta_{w}^{\alpha,d}$. We use this inequality to establish a new version of Young's type inequality.
\end{abstract}
\noindent{\bf keywords.}{  Weinstein transform, Weinstein operator, Babenko inequality, Young's type inequality}

\section{Introduction}
%%%%%%%%%%%%%%%%%%%%%%%%%%%%%%%%%%%%%%%%%%%%%%%%%%%%%%%%%%%%%%%%%%%%%%%%%%%%%%%%%%%%%%%%%%%%%
%%%%%%%%%%%%%%%%%%%%%%%%%%%%%%%%%%%%%%%%%%%%%%%%%%%%%%%%%%%%%%%%%%%%%%%%%%%%%%%%%%%%%%%%%%%%%
Inequalities are a basic tool in the study of the classical Fourier analysis. The result relating $L^{p}$ estimates for a function and its Fourier transform is the Hausdorff-Young inequality (see\cite{B.a}) it given for $ f \in L^{p}(\mathbb{R}^{n})$, $1 < p \leq 2$ and $q$ such that $ \frac{1}{p} + \frac{1}{q} = 1$,
 %The Hausdorff-Young inequality is one of the most fundamental results about the mapping of the Fourier transform  for $ f \in L^{p}(\mathbb{R}^{n})$, $1 < p \leq 2$ and q such that $ \frac{1}{p} + \frac{1}{q} = 1$, it given by:(see\cite{B.a})
$\|\hat{f}\|_{q} \leq  \|f\|_{p}$
where $\hat{f}$ is the classical Fourier transform defined by:
\begin{equation}\label{e5453}
\hat{f}(\lambda)= \int_{\mathbb{R}^{n}} e^{-2 i \pi <\lambda,x>} f( x ) dx.
\end{equation}
This sharp form of the Hausdorff-Young inequality was  extended by W.Bechner \cite{B.K}   in the form:
%has established the following inequality  t
%for $ f \in L^{p}(\mathbb{R}^{n})$, $1 < p \leq 2$ and q such that $ \frac{1}{p} + \frac{1}{q} = 1$
 \begin{equation}\label{e1.1}
 \| \hat{f}\|_{q} \leq  \left( \frac{p^{\frac{1}{p}}  }{q^{ \frac{1}{q}}}\right) ^{\frac{n}{2}} \|f\|_{p}.
\end{equation}
Next, using the method as W.Bechner, A.Fitouhi \cite{F.A} in 1975 proved that the Fourier-Bessel transform $\mathcal{F}(f)$ of $ f \in L^{p}(\mathbb{R}_{+}, \frac{x^{2\alpha +1}}{ 2^{\alpha} \Gamma(\alpha + 1)} dx) $ satisfies the following inequality:
\begin{equation}\label{e2}
 \|\mathcal{F}(f)\|_{q,\alpha} \leq  \left( \frac{p^{\frac{1}{p}}  }{q^{ \frac{1}{q}}}\right)^{\alpha + 1}  \|f\|_{p,\alpha},\quad \alpha > -\frac{1}{2}, % for  1 < p \leq 2 %and q such that $ \frac{1}{p} + \frac{1}{q} = 1$
\end{equation}
where $\mathcal{F} $ is given by:
\begin{equation}\label{e3}
 \mathcal{F}( f ) (\lambda ) = \frac{1}{ 2^{\alpha} \Gamma(\alpha + 1)} \int_{0}^{ + \infty} f(x) j_{\alpha}( \lambda x )x^{2\alpha +1}dx,
\end{equation}
and $j_{\alpha}$ is the normalized Bessel function of index $\alpha$, defined by:
 \begin{equation}\label{e3}
    j_{ \alpha} ( z ) = \Gamma ( \alpha +1 ) \sum _{ n=0} ^{ +\infty}     \frac{(-1) ^{n} }{ n ! \Gamma ( n + \alpha + 1 )}  (\frac{z}{2} )^{2 n}. \quad\forall z \in \mathbb{C} .
 \end{equation}
 Moreoer, F. Bouzeffour \cite{F.B} in 2014 adapted these ideas and he applied them to obtain the Babenko-Bechner-type inequality for the Dunkl transform in real line.\\
% In this paper,
  %This operator has several applications in pure and applied mathematics, especially in fluid mechanics.\\
  % In this paper, similar to arguments given by Bechner \cite{B.K} and Fitouhi \cite{F.A}, we establish a new Babenko-Bechner-type inequality for the Weinstein transform.
   \quad In \cite{Z.B.NA,N.B.T}, Ben Salem.N and Ben Nahia.Z was defined the Weinstein operator $ \Delta_{w}^{\alpha,d}$  on $ \mathbb{R}_{+}^{d+1} = \mathbb{R}^{d} \times [0,+\infty [ $ by
\begin{equation}\label{e4}
 \Delta_{w}^{\alpha,d} = \Delta_{d}+\mathcal{L}_{\alpha},\quad\alpha > -\frac{1}{2},
\end{equation}
  where $\Delta_{d}$ is the ordinary Laplacian on $\mathbb{R}^{d}$  and $\mathcal{L} _{\alpha}$ is the Bessel operator for the last variable $ x_{d+1}$.\\
  The Weinstein operator  $ \Delta_{w}^{\alpha,d}$  has several applications in pure and applied mathematics, especially in
fluid mechanics \cite{B.M}.
The Weinstein transform generalizing the usual Fourier transform, is given for $  f \in L _{\alpha} ^{1} ( \mathbb{R} _{+} ^{d + 1}) $ and $\lambda \in \mathbb{C}_{+}^{d+1}$, by
$$  \mathcal{F} _{w} ^{\alpha ,d } ( f ) ( \lambda )  =  \int _{\mathbb{R} _{+} ^{d + 1} }  f(x ) \Psi _{ \alpha , d } ( x , \lambda ) d\nu_{\alpha ,d}(x),$$ where   $ d\nu_{\alpha ,d}(x)$ is the measure on $\mathbb{R}_{+}^{d+1}$ and $  \Psi _{ \alpha , d }$
is the Weinstein kernel and  given respectively later
by (\ref{e42}) and (\ref{e5}).
%$$\mathcal{L} _{\alpha} = \frac{\partial^{2}}{\partial x_{d+1}^{2}} + \frac{2\alpha + 1}{x_{d+1}} \frac{\partial }{\partial x_{d+1}
% }.$$
%During the last few years a gigantic investigations and efforts have been provided by several
%authors to study the Weinstein operator (\cite{A.S,M.K,A.S.I}).
The purpose of the present paper is to establish  the Babenko-Bechner-type inequality for the Weinstein transform associated to the the Weinstein operator.
The paper is organized as follows: In the first section we recall some basic Harmonic results related with the differential operator $\Delta_{w}^{\alpha,d}$ given by (\ref{e4}). In the second section we establish a Babenko-Bechner-type inequality for the Weinstein transform. As application, new Young's type inequality is given.

\section{Preliminaries}
%The Harmonic analysis associated with the Weinstein transform is studied by Ben Nahia, Ben Salem in (\cite{N.B.T,Z.B.NA}).
 In this section, we collect some notations and results related to the Weinstein operator $ \Delta_{w}^{\alpha,d}$ (see \cite{Z.B.NA,N.B.T}) .\\
  %\textbf{Notation}:
  In the following we denote by
  \begin{itemize}
  \item $\mathbb{R}_{+}^{d+1}= \mathbb{R}^{d}\times [0,+\infty [$
  \item $ x=(x_{1},\cdots,x_{d} , x_{d+1}) $;\quad $ x^{\prime} =( x_{1},\cdots,x_{d} ) $
  \item $\|x\|= \sqrt{x_{1}^{2}+\cdots + x_{d+1}^{2}}$
    \item $\mathcal{C}_{*}( \mathbb{R}^{d+1})$, the space of continuous functions on $  \mathbb{R}^{d+1}$, even with respect to the last variable.
 \item  $\mathcal{S}_{*}(\mathbb{R}^{d+1})$, the space of the $C ^{\infty}$functions, even with respect to the last variable and rapidly decreasing together with their derivatives.
          \item $\mathcal{D}_{*}( \mathbb{R}^{d+1})$, the space of the $C ^{\infty} -$ functions on $  \mathbb{R}^{d+1}$ which are of compact support, even with respect to the last variable.

 \item $ L_{\alpha}^{p} (\mathbb{R}_{+}^{d+1}) $, $ 1\leq p \leq +\infty $, the space of measurable functions $ f $ on $  \mathbb{R}_{+}^{d+1}$ such that:
 \begin{eqnarray*}
 \|f\|_{\alpha,p}& =& \left( \int_{\mathbb{R}_{+}^{d+1}} |f ( x )|^{p}d\nu_{\alpha,d}( x ) \right)^{\frac{1}{p}}  < + \infty,\quad if \quad p \in [1, + \infty[,\\
 \|f\|_{\alpha,\infty}& = &ess \sup _{x \in \mathbb{R}_{+}^{d+1}} |f ( x )|  < + \infty,
 \end{eqnarray*}
where $ d\nu_{\alpha,d}$ is the measure  on $\mathbb{R}_{+}^{d+1}$ given by:
      \begin{equation}\label{e42}
           d\nu_{\alpha ,d} ( x ) = \frac{x _{d+1} ^{ 2 \alpha+ 1}   }{ ( 2\pi )^{     \frac{d}{2}}  2 ^{\alpha} \Gamma ( \alpha + 1)} dx.
           \end{equation}
          \end{itemize}
%For all  $  \lambda =(\lambda_{1},\lambda_{2}, \cdots ,\lambda_{d+1}) \in \mathbb{C}^{d+ 1}$, the system :

%$$ \frac{\partial^{2}U}{\partial x^{2}_{j}}(x) = -\lambda_{j}^{2} U(x), \quad \quad if \quad 1\leq j\leq d $$

%\begin{equation}\label{e44}
% \mathcal{L}_{\alpha} U(x) = -\lambda_{d+1}^{2} U(x),\quad
%\end{equation}

%$$ U(0) = 1,\quad \frac{\partial U}{\partial x_{d+1}} = 0 \quad and \quad \frac{\partial U}{\partial
                                                                    %  x_{j}}(0) = -i\lambda_{j},  \quad \quad if \quad 1\leq j\leq
The Weinstein Kernel $ \Psi _{ \alpha , d} ( x, . )$ is given by:
\begin{equation}\label{e5}
  \Psi _{ \alpha , d} ( x , z ) = e^{- i <  x^{ \prime} ,  z ^{\prime} > }  j_{ \alpha} (x _{ d+ 1} z _{d+1} )\quad \forall (x,z)\in \mathbb{R} ^{ d + 1}\times\mathbb{C} ^{ d + 1}.
  \end{equation}
 %  where  $ z = ( z ^{\prime}, z _{d+ 1} );\quad z ^{\prime} = ( z_{1}, z_{2},..., z_{d} ) $; \quad $ x=( x^{\prime}, x_{d+1}) $;\quad $ x^{\prime} =( x_{1},\cdots,x_{d} ) $

  The function $( x, z ) \mapsto \Psi _{ \alpha , d} ( x, z )$ has a unique extention to $  \mathbb{C} ^{ d + 1}\times \mathbb{C} ^{ d + 1} $
   and has the following integral operator:
 \begin{equation}\label{e1515}
  \Psi _{ \alpha , d}( x , z )= \frac{\Gamma(\alpha+1)}{ \sqrt{\pi}\Gamma(\alpha+\frac{1}{2})}e^{-i<x^{\prime}; z^{\prime}>}\int_{-1}^{1}(1-u^{2})^{\alpha-\frac{1}{2}}e^{-i x_{d+1}z_{d+1}u}du.
 \end{equation}

The Weinstein Kernel %$\Psi _{ \alpha , d}$
satisfies the following properties:% see (\cite{N.B.T},  \cite{Z.B.NA})
\begin{itemize}
  \item For all $(x,  z ) \in\mathbb{C} ^{ d + 1} \times \mathbb{C} ^{ d + 1}$ we have:
 $$ \Psi _{ \alpha , d } ( x , z) =   \Psi _{ \alpha , d } ( z , x); \quad\Psi _{ \alpha , d } (x, 0 ) = 1 .  $$
  \item For all $(x,  z ) \in\mathbb{C} ^{ d + 1} \times \mathbb{C} ^{ d + 1}$ and $\lambda  \in\mathbb{C}$ we have:
  $$\Psi _{ \alpha , d } (\lambda x , z) = \Psi _{ \alpha , d } ( x , \lambda z).$$
 %  $$ \Psi _{ \alpha , d } ( x , - z) =   \Psi _{ \alpha , d }( -x , z). $$
 % \item For all $(x,  z ) \in\mathbb{C} ^{ d + 1} \times \mathbb{C} ^{ d + 1}$ we get:

  \item For all $ v \in \mathbb{ N} ^{ d +1} $;  $ x \in\mathbb{R}_{+} ^{d + 1} $ and  $ z \in \mathbb{ C }^{ d+1}$:
   \begin{equation}\label{e6}
   | D_{z} ^{v}  \Psi _{ \alpha , d } ( x , z ) |  \leq \| x \|^{|v|} exp \left(  \| x \| \| Img( z )\| \right),
   \end{equation}
  where
   $$ D_{z} ^{v} = \frac{\partial ^{v}}{ \partial z_{1} ^{v_{1}} ...\partial z_{d+1} ^{v_{d+1}}},$$
   and $$ | v| = v_{1} + v_{2}+ ...+ v_{ d + 1}.  $$
   \item
   $ \forall x , y \in\mathbb{R}  _{+} ^{ d+ 1}  $, we have:
   \begin{equation}\label{e7}
   | \Psi _{ \alpha , d } ( x , y) | \leq 1.
   \end{equation}
  \end{itemize}

% In the following, we denote by :

   \begin{definition}
The Weinstein transform is given for $  f \in L _{\alpha} ^{1} ( \mathbb{R} _{+} ^{d + 1} ,d\nu_{\alpha,d}(x) ) $  by:
 \begin{equation}\label{e49}
      \mathcal{F} _{w} ^{\alpha ,d } ( f ) ( \lambda )  =  \int _{\mathbb{R} _{+} ^{d + 1} }  f(x ) \Psi _{ \alpha , d } ( x , \lambda ) d\nu_{\alpha ,d}(x),\quad \forall \lambda \in \mathbb{R} _{+} ^{d + 1}.
       \end{equation}
\end{definition}
We list some known basic properties of the Weinstein transform:
\begin{itemize}
    \item For $ f \in L _{\alpha} ^{1} ( \mathbb{R} _{+} ^{d + 1},d\nu_{\alpha,d}(x) ) $, we have:
     \begin{equation}\label{e50}
         \| \mathcal{F} _{w}^{ \alpha ,d} ( f ) \| _{\alpha , \infty} \leq  \| f \| _{ \alpha, 1}.
      \end{equation}
      \item Let  $ f \in L _{\alpha} ^{1} ( \mathbb{R} _{+} ^{d + 1},d\nu_{\alpha,d}(x) ) $ such that the function $ \mathcal{F} _{w}^{ \alpha ,d}(f)$  $\in  L _{\alpha} ^{1} ( \mathbb{R} _{+} ^{d + 1},d\nu_{\alpha,d}(x) )$, we have the following inversion formula:
          \begin{equation}\label{e53}
  f( x )= \int _{\mathbb{R}_{+} ^{d +1}}   \mathcal{F} _{w} ^{\alpha,d} ( f ) ( y ) \Psi _{\alpha,d} (- x,  y) d\nu_{ \alpha,d}(y).
 \end{equation}
 \item For all $f$, $g $ $\in \mathbb{S}_{*}(\mathbb{R}_{+}^{d+1} )$, we have the Parseval formula:
 \begin{equation}\label{e55}
 \int _{ \mathbb{R}_{+} ^{ d +1}  }  f ( x ) \overline{g} ( x ) d\nu_{ \alpha,d}(x ) =   \int _{ \mathbb{R}_{+} ^{ d +1}  } \mathcal{F} _{w} ^{\alpha,d}( f) ( \lambda ) \overline{ \mathcal{F} _{w} ^{\alpha,d} ( g )} ( \lambda ) d\nu_{ \alpha,d}( \lambda ).
 \end{equation}
 \item For every $ f \in L _{\alpha} ^{2} (\mathbb{R}_{+} ^{ d +1};d\nu_{ \alpha,d}(x )  )$, we have the Plancherel formula:
  \begin{equation}\label{e505}
   \int _{ \mathbb{R}_{+} ^{ d +1}  }   |f(x)|^{2} d\nu_{ \alpha,d}( x) = \int _{ \mathbb{R}_{+} ^{ d +1}  } |  \mathcal{F}_{w} ^{\alpha,d}(f)(\lambda)|^{2} d\nu_{ \alpha,d}( \lambda).
\end{equation}
\end{itemize}
%\end{proposition}
\begin{definition}
The generalized translation operator $T_{x}^{\alpha,d}$ ,  $ x \in \mathbb{R}_{+}^{d+1} $ associated with the Weinstein operator $ \Delta_{ w} ^{\alpha,d} $ is defined for $ y \in \mathbb{R} _{+}^{d +1} $ and $  f \in\mathcal{C}_{*} (\mathbb{R} _{+}^{d +1}) $  by:
\begin{equation}\label{e59}
   T_{x}^{\alpha,d} f (y ) =\frac{a _{\alpha}}{2} \int_{0} ^{\pi} f \left( x^{ \prime }+ y ^{\prime  },\sqrt{ x_{d+1}^{2} +y_{d+1}^{2 } +2 x_{d+1} y_{d+1}\cos \theta } \right)  (\sin \theta )^{ 2\alpha} d\theta,
  \end{equation}
where $ a_{\alpha} = \frac{\Gamma(\alpha +1)}{ \sqrt{\pi}\Gamma(\alpha +\frac{1}{2})}$.
\end{definition}
\begin{definition}
  The convolution product of  $ f $, $g$ $
  \in L_{\alpha}^{1}( \mathbb{R}_{+}^{d+1}; d\nu_{ \alpha,d})$ is defined for every $ x \in \mathbb{R}_{+}^{d+1} $ by:
     \begin{equation}\label{e519}
     f*_{w}g (x)=\int_{\mathbb{R}_{+}^{d+1}} T_{x}^{\alpha,d}f(y) g(y) d\nu_{\alpha,d}(y).
     \end{equation}
     \end{definition}
     The convolution product satisfies the following properties:
     \begin{itemize}
     \item For all $f$, $g$ $ \in L_{\alpha}^{1}( \mathbb{R}_{+}^{d+1}; d\nu_{ \alpha,d})$, $f*_{w}g \in L_{\alpha}^{1}( \mathbb{R}_{+}^{d+1}; d\nu_{ \alpha,d}) $, and we have:
  \begin{equation}\label{e159}
  \mathcal{F}_{w}^{\alpha,d}(f *_{w} g ) = \mathcal{F}_{w}^{\alpha,d}(f)\mathcal{F}_{w}^{\alpha,d}(g).
  \end{equation}
  \item The function $\Psi _{\alpha,d} (x , \lambda )$ satisfies the following product formula:
  \begin{equation}\label{e61}
    \forall y \in \mathbb{R} _{+} ^{d+1} ,\Psi _{\alpha,d} (x , \lambda ) \Psi _{\alpha,d} (y , \lambda )= T_{x}^{\alpha,d}\left[ \Psi _{\alpha,d} (. , \lambda ) \right] (y).
  \end{equation}
  \item For every $ f \in L_{\alpha}^{1}( \mathbb{R}_{+}^{d+1}; d\nu_{ \alpha,d})$,
  \begin{equation}\label{e62}
    \mathcal{F}_{w} ^{ \alpha ,d } \left( T_{x}^{\alpha,d} f \right) = \Psi _{\alpha,d} (x , y)  \mathcal{F}_{w} ^{ \alpha ,d } ( f ) (y ).
\end{equation}
\item Let $ p,q, r \in [1,+\infty [ $ such that  $\frac{1}{p} +\frac{1}{q}-\frac{1}{r} = 1$, if  $ f \in L_{\alpha}^{p} (\mathbb{R}_{+}^{d+1}; d\nu_{ \alpha,d})$,\\ $   g \in L_{\alpha}^{q} (\mathbb{R}_{+}^{d+1}; d\nu_{ \alpha,d})$, then $f*_{w}g  \in  L_{\alpha}^{r}(\mathbb{R}_{+}^{d+1}; d\nu_{ \alpha,d})$ and we have:
     \begin{equation}\label{e434}
     \|f*_{w}g\|_{r,\nu_{ \alpha,d} }\leq \|f\|_{p,\nu_{ \alpha,d}}\|g\|_{q,\nu_{ \alpha,d}}.
     \end{equation}
 \end{itemize}

\section{Babenko-type inequality for the Weinstein transform}
 The following result is a consequence of the central limit theorem \cite{K.T} to obtain  the gaussian mesure $d\mu_{\alpha,d}(x)$, for this we
define the Weinstein convolution for two positives and bounded measures $\gamma_{1}$ and $\gamma_{2}$ in $\mathbb{R}_{+}^{d+1}$ and $ for f \in \mathcal{D}_{*}( \mathbb{R}^{d+1})$ by:
 \begin{equation}\label{e602}
 \gamma_{1} *_{w} \gamma_{2}( f ) = \int_{\mathbb{R}_{+}^{d+1}}\int_{\mathbb{R}_{+}^{d+1}} T_{x}^{\alpha,d}f( y ) d \gamma_{1}(x)d\gamma_{2}(y).
\end{equation}
%\end{definition}
%\begin{proposition}
  Inspired by the work of \cite{ B.K,F.A}, we consider the following measures $\beta$ and $\beta_{n}(x)$ given by:
 \begin{equation}\label{e1602}
 \beta = \otimes_{j=1}^{d+1} \beta_{j},
 \end{equation}
 with \begin{center}
  $ \left\{
 \begin{array}{ccc}
 \beta_{j}=    \frac{1}{2}( \delta_{-1} + \delta_{1} )\quad  1 \leq j\leq d,\\
 \beta_{d+1} =  \frac{1}{2} \left(\delta_{0} + \delta_{\sqrt{4\alpha + 4}}\right)\quad,

                                                          \end{array}\right.$ \
                                                          $$$$
                                                          and for  $x=(x_{1},..,x_{d},x_{d+1}) \in \mathbb{R}_{+}^{d+1} $;  $ n=(n_{1},...,n_{d+1}) \in \mathbb{N}^{d+1}$,
                                                          \end{center}
\begin{equation}\label{e6092}
   \beta_{n}(x)=\beta(\sqrt{n_{1}}x_{1},...,\sqrt{n_{d+1}}x_{d+1})=\otimes_{j=1}^{d+1} \beta_{j}(\sqrt{n_{j}}x_{j}),\quad n_{j}\in \mathbb{N} \setminus \{0\}.
   \end{equation}
The n-fold convolution of  $ \gamma_{n} $  is given by:
\begin{equation}\label{e2}
  \gamma_{n} = \beta_{n} *_{w} \cdots *_{w}  \beta_{n},\quad n=(n_{1},...,n_{d+1}) \in \mathbb{N}^{d+1}.
 \end{equation}
% \end{proposition}

%\begin{corollary}

Then the sequence of measure  $ \{\gamma_{n}\}_{n}$ defined in (\ref{e2}) converge weakly to
 \begin{equation}\label{e162}
 d\mu_{\alpha,d}(x) =\frac{x _{d+1} ^{ 2 \alpha+ 1}   }{ ( 2\pi )^{     \frac{d}{2}}  2^{\alpha} \Gamma ( \alpha + 1)} e^{-\frac{\|x\|^{2}}{2}} dx.
 \end{equation}
In what follows we use the standard multi-index notation.  $$ m!=m_{1}!\cdots m_{d+1}!,\quad for \quad m = (m_{1},\cdots m_{d+1}) \in \mathbb{N}^{d+1},$$  and
  $$ x^{n} =x_{1}^{ n_{1}}\cdots x_{d+1}^{ m_{d+1}}
  ,\quad for \quad x=(x_{1},\cdots,x_{d+1}).$$
%$\mathcal{\ell}$
\subsection{Multidimensional Hermite polynomials}
 % In this subsection, we give some results related to the Multidimensional Hermite polynomials $\{H_{m}^{\alpha,d}(x)\} $ corresponding to the Gaussian Measure given by (\ref{e162}).\\
 We consider the system of multi-dimensional Hermite polynomial (see \cite{S.G,R.O}):
 \begin{equation}\label{e001}
H_{m}^{\alpha,d}(x)= \left(\prod _{k=1}^{d}H_{m_{k}}(x_{k})\right)\times  \ell_{m_{d+1}}^{\alpha}(x
_{d+1}),
\end{equation}
where $m= (m_{1},\cdots,m_{d+1}) \in \mathbb{N}^{d+1}$,  $ x=(x_{1},x_{2}\cdots x_{d+1})  \in \mathbb{R}_{+}^{d+1}$, $H_{m_{k}}(x_{k})$ are the one-dimensional Hermite polynomial, $\{k=1\cdots d \}$ and $  \ell_{m_{d+1}}^{\alpha}$ is given by:$$  \ell_{m_{d+1}}^{\alpha}(x_{d+1}) = 2^{m_{d+1}} m_{d+1}! L_{m_{d+1}}^{\alpha}(\frac{x_{d+1}^{2}}{2}), $$ where $ L_{m_{d+1}}^{\alpha}(x_{d+1}) $ represents the Laguerre polynomial of index $\alpha\geq\frac{-1}{2} $.\\
%Let $ (n,m) \in  \mathbb{N}^{d+1} \times \mathbb{N}^{d+1}$,  $x=(x_{1},\cdots x_{d+1})\in \mathbb{R}_{+}^{d+1} $,
  The  polynomials $ H_{m}^{\alpha,d}$ satisfy  the orthogonality relation:
\begin{equation}\label{e3ww6}
\int_{\mathbb{R}_{+}^{d+1}}H_{n}^{\alpha,d}(x)H_{m}^{\alpha,d}(x)d\mu_{\alpha,d}(x)=\gamma_{m}^{-1} \prod_{i=1}^{d+1} \delta_{n_{i},m_{i}},
\end{equation}
where
\begin{equation}\label{e3ww688}
\gamma_{m}= \prod_{k=1}^{d}\frac{\gamma_{m_{d+1}}}{m_{k}!},
\end{equation}
with

\begin{equation}\label{e15001}
 \gamma_{m_{d+1}}  = \frac{ (-1)^{m_{d+1}} \Gamma(\alpha+1)}{ m_{d+1}! \quad 2^{2m_{d+1}}  \Gamma(m_{d+1}+\alpha+1)},
 \end{equation}and
  $\delta_{n_{i},m_{i}}$ is the Kronecker delta.
\begin{proposition}
Let %$x=(x_{1}\cdots x_{d+1} )$,
 $t=(t_{1},\cdots t_{d+1}) \in \mathbb{R}_{+}^{d+1}$ and $m= (m_{1} \cdots,m_{d+1}) \in \mathbb{N}^{d+1}$, the generating function of $\{H_{m}^{\alpha,d}(x)\}$ is given by:%defined (\ref{e001}) is given by:
\begin{equation}\label{eww5}
   e^{-\frac{\|t\|^{2}}{2}}\Psi _{ \alpha , d} (ix , t )=\sum_{m_{1},\cdots,m_{d +1}}^{0, \infty}\gamma_{m}H_{m}^{\alpha,d}(x) t^{m} t_{d+1}^{m_{d+1}},
\end{equation}
where $\gamma_{m}$ is defined in (\ref{e3ww688}).

\end{proposition}
\begin{proof}
For $t\in \mathbb{R}_{+}^{d+1}$, we consider the function:
 \begin{eqnarray*}
 f(t)&=&e^{-\frac{\|t\|^{2}}{2}}\Psi _{\alpha, d} (ix, t)\\
  &=&\left[\prod_{k=0}^{d} e^{-\frac{t_{k}^{2}}{2}} e^{ x _{k} t _{k} }\right] \left[ e^{-\frac{(t_{d+1})^{2}}{2}}j_{ \alpha} ( ix_{d+1}t _{ d+ 1} )\right].
  \end{eqnarray*}
Since $f$ is the product of the $d+1$ functions which have developable in integer series then
 $$ f(t) = \sum_{m_{1},..,.m_{d +1}=0}^{\infty}\left[   \frac{\partial^{|m|}}{\partial t_{1}^{m_{1}}... \partial t_{d+1}^{m_{d+ 1}}} e^{-\frac{\|t\|^{2}}{2}}\Psi _{ \alpha , d} (ix , t ) \right]_{|t|=0}  \frac{t_{1}^{m_{1}}}{m_{1}!}...\frac{t_{d}^{m_{d}}}{m_{d}!} \frac{t_{d+1}^{m_{d+1}}}{m_{d+1}!}.$$
Now, from the generating functions of $H_{m_{k}}(x _{k})$ and  the generating functions of $\ell_{m_{d+1}}^{\alpha}(x_{d+1})$%, (see \cite{F.A,B.K,F.B}
, we get:
 $$\left[   \frac{\partial^{|m|}}{\partial t_{1}^{m_{1}}... \partial t_{d+1}^{m_{d+ 1}}} e^{-\frac{\|t\|^{2}}{2}}\Psi _{ \alpha , d} (ix , t ) \right]_{|t|=0}$$
 \begin{eqnarray*}
  &=&\prod_{k=0}^{d}\left( \frac{\partial^{m_{k}}}{\partial t_{k}^{m_{k}}} e^{-\frac{t_{k}^{2}}{2}} e^{ x _{k} t _{k} } \right)_{t_{k}=0}\times \left(\frac{\partial^{m_{d+1}}}{\partial t_{d+1}^{m_{d+1}}}(e^{-\frac{(t_{d+1})^{2}}{2}}j_{\alpha}(ix_{d+1}t _{ d+ 1} ))\right)_{t_{d+1}=0}\\
%From the relation (\ref{eA0}), we get:
%$$\left[   \frac{\partial^{|m|}}{\partial t_{1}^{m_{1}}... \partial t_{d+1}^{m_{d+ 1}}} e^{-\frac{\|t\|^{2}}{2}}\Psi _{ \alpha , d} (ix , t ) \right]_{t=0}
 % =$$
  &=&  \frac{ (-1)^{m_{d+1}} \Gamma(\alpha+1)}{
 \quad 2^{2m_{d+1}}  \Gamma(m_{d+1}+\alpha+1)}\left(\prod_{k=0}^{d}  H_{m_{k}}(x_{k})\right) \times\ell_{m_{d+1}}^{\alpha}(x_{d+1}).
 \end{eqnarray*}
 %We pose:
 %$$ H_{m_{1},...,m_{d +1}}^{\alpha,d}(x)= \left(\prod _{k=1}^{d}H_{m_{k}}(x_{k})\right)\times  \ell_{m_{d+1}}^{\alpha}(x
%{d+1})\quad. $$

\end{proof}

\begin{definition}
For $|w|<1$, we define the integral operator $K_{w}(f)$ by:
\begin{equation}\label{e25}
K_{w}(f)(y)=\int_{\mathbb{R}_{+}^{d+1}}K_{w}(x,y)f(x)d\mu_{\alpha,d}(x),
\end{equation}
where $K_{w}(x,y)$ is the Mehler Kernel-type given by:
\begin{equation}\label{e26}
K_{w}(x,y)=
\frac{1}{(1-w^{2})^{\alpha+ 1 +\frac{d}{2}}}  exp\left[\frac{-w^{2}(\|x\|^{2}+\|y\|^{2})}{2(1-w^{2})}\right]\Psi_{\alpha,d}\left(\frac{-iw}{1-w^{2}}x , y \right).
\end{equation}
\end{definition}
 \begin{proposition}
Let $ n=(n_{1},\cdots,n_{d+1})\in  \mathbb{N}^{d+1}$,  for $|w|<1$, we have:
%the integral operator  given by the relation (\ref{e25})  satisfy the relation as
\begin{equation}\label{e27}
K_{w}(H_{n}^{\alpha,d})(y) %\int_{\mathbb{R}_{+}^{d+1}}K_{w}(x,y)H_{n_{1},\cdots,n_{d+1}}^{\alpha,d}(x) d\mu_{\alpha,d}(x)
= w^{|n|+n_{d+1}} H_{n}^{\alpha,d}(y)  .
\end{equation}
\end{proposition}
\begin{proof}
%The relation (\ref{e26}) of  the Mehler Kernel-type  $K_{w}(x,y)$, allows us to write:
%\begin{eqnarray*}
%\int_{\mathbb{R}_{+}^{d+1}}K_{w}(x,y)H_{n}^{\alpha,d}(x) d\mu_{\alpha,d}(x)&=&\int_{\mathbb{R}_{+}^{d+1}}\frac{exp\left[\frac{-w^{2}(\|x\|^{2}+\|y\|^{2})}{2(1-w^{2})}\right]}{(1-w^{2})^{\alpha+ 1 +\frac{d}{2}}}\Psi_{\alpha,d}\left(\frac{-iwx}{1-w^{2}},y \right) H_{n_{1},...,n_{d+1}}^{\alpha,d}(x) d\mu_{\alpha,d}(x)\quad,
%\end{eqnarray*}
%From the relations (\ref{e5}), (\ref{e001}) and (\ref{e26}), we have :
%$$\int_{\mathbb{R}_{+}^{d+1}}K_{w}(x,y)H_{n_{1},...n_{d+1}}^{\alpha,d}(x) d\mu_{\alpha,d}(x)$$ $$=\left( \prod_{j=1}^{d}\int_{\mathbb{R}}(1-w^{2})^{\frac{-1}{2}} exp\left[\frac{-w^{2}(x_{j}^{2}+y_{j}^{2})}{2(1-w^{2})} - \frac{w}{1-w^{2}}x_{j}y_{j} \right] H_{n_{j}}(x_{j}) \frac{1}{ \sqrt{2\pi}}e^{-\frac{x^{2}}{2}}dx_{j}\right)$$$$\times\left(\int_{\mathbb{R}_{+}}\frac{ exp\left[\frac{-w^{2}}{2(1-w^{2})}(x_{d+1}^{2}+y_{d+1}^{2})\right] } {(1-w^{2})^{\alpha+1}} j_{\alpha}\left(\frac{iw}{1-w^{2}}x_{d+1}y_{d+1}\right)\ell_{n_{d+1}}^{\alpha}(x_{d+1})\frac{e^{-\frac{x^{2}}{2}}}{2^{\alpha} \Gamma(\alpha +1)}x_{d+1}^{2\alpha+1}dx_{d+1}\right).$$
%By Beckner (\cite{B.K}) and Fitouhi (\cite{F.A}), we get:
%\begin{eqnarray*}
%\int_{\mathbb{R}_{+}^{d+1}}K_{w}(x,y)H_{n_{1},...n_{d+1}}^{\alpha,d}(x) d\mu_{\alpha,d}(x)&=&\left(\prod_{j=1}^{d} w^{n_{j}}H_{n_{j}}(y_{j})\right) \times w ^{2n_{d+1}}\ell_{n_{d+1}}^{\alpha}(x_{d+1})\\
%&=& w^{n_{1}+... +2n_{d+1}}H_{n_{1},...n_{d+1}}^{\alpha,d}(y)\\
%&=& w^{|n|+n_{d+1}} H_{n_{1},...n_{d+1}}^{\alpha,d}(y).
%\end{eqnarray*}
From the the expression of the Mehler Kernel-type $K_{w}(x,y)$ defined in (\ref{e26}) and the polynomials $ H_{m}^{\alpha,d}(x)$ defined in (\ref{e001}) , we can easily find the result.
\end{proof}
\begin{lemma}\label{e000090}
Let $ n,m \in \mathbb{N}$, $ x=(x_{1},x_{2}\cdots, x_{n})$, $u=(u_{1},u_{2} \cdots, u_{n}) \in \mathbb{R}^{n}$, the Hermite polynomial satisfies the following relation:
 $$ H_{m}(x_{1}u_{1}+...+x_{n}u_{n})=\sum_{s=0}^{m}\frac{2^{m}m!}{(m-s)!}h_{s}(x_{1}(u_{1}-1),\cdots,x_{n}(u_{n}-1))H_{m-s}(x_{1}+\cdots+x_{n}),$$
where the homogenous symmetric polynomial $ h_{s}(y) $ are given by $$h_{0}(y)=1\quad and \quad
 h_{s}(y)=\sum_{|k|=s}\frac{y^{k}}{k!};\quad y\in \mathbb{R}^{n}.$$
\end{lemma}\label{e0000054}
\begin{proof}
For the proof, we refer \cite{F.B}.
%For more details, see (\cite{F.B}).
\end{proof}
\begin{lemma}
 If $x_{i}=(x_{i_{1}},...,x_{i_{d+1}}) \in  \mathbb{R}^{d+1}$ for $1 \leq i \leq n$, satisfies the following conditions:\\
 \begin{center}
 $ \left\{
 \begin{array}{ccc}\label{4A}
                                                                   x_{1_{j}}^{2} = x_{2_{j}}^{2} ... = x_{n_{j}}^{2} =\frac{1}{n_{j} }, \quad 1\leq j\leq d,\\
                                                                       x_{1_{d+1}}^{2} =...= x_{n_{d+1}}^{2} = \frac{4\alpha +4}{n_{d+1}}\quad or\quad 0, \\
                                                                   \end{array}\right.$ \
                                                                   \end{center}

then, for every $ m=(m_{1},\cdots,m_{d+1})\in \mathbb{N}^{d+1}$
, such that $ d+1\leq |m|\leq |n|$  we have:
\begin{equation}\label{e28}
\gamma_{m_{d+1}}  T_{x_{1}}^{\alpha,d}\circ T_{x_{2}}^{\alpha,d}\circ ...\circ T_{x_{n-1}}^{\alpha,d} H_{m}^{\alpha,d}(x_{n}) = \mathcal \phi_{m,n}^{\alpha,d}(x_{1},...x_{n})  + \mathcal P_{m,n}^{\alpha,d}(x_{1},...x_{n}),
\end{equation}
 where $\gamma_{m_{d+1}}$ is defined in (\ref{e15001}),  $\mathcal P_{m,n}^{\alpha,d}(x_{1},...x_{n}) $ is a polynomial of degree less than $ |m|-1$ such that
$$ P_{m,n}^{\alpha,d}(x_{1},...x_{n}) \rightarrow 0,\quad n_{j} \rightarrow \infty, \quad  \forall \quad 1\leq j \leq d+1, $$
 and $\phi_{m,n}^{\alpha,d}(x_{1},...x_{n})$ is a homogeneous symmetric polynomial of degree $|m|$ which is defined as follow
 \begin{equation}\label{00A}
\phi_{m,n}^{\alpha,d}(x_{1},..,x_{n})=
\end{equation}
$$\left[\prod_{k=1}^{d} m_{k}!\sigma_{n_{k},m_{k}}(x_{1_{k}},..,x_{n_{k}})\right]\left[ \sum_{s=0}^{m_{d+1}} 2^{m_{d+1}} \tilde{h_{s}}(x_{1_{d+1}},..,x_{n_{d+1}})\sigma_{n_{d+1} m_{d+1}-s}(x_{1_{d+1}},...,x_{n_{d+1}})\right]
$$
where
$$\sigma_{n_{i}l}(x_{1_{i}},..,x_{n{i}}) = \sum_{1\leq j_{1} <...< j_{l} \leq n_{i}   }   x_{j_{1}}...x_{j_{l}}\quad 1\leq i \leq d+1,$$
and
 $$ \tilde{h_{s}}(x_{1_{d+1}},...,x_{n_{d+1}}) =(-1)^{s}\sum_{|k|=s}\frac{1}{k!}C_{k_{1_{d+1}},\cdots,k_{n_{d+1}}}x_{1_{d+1}}^{k_{1_{d+1}}}\cdots x_{n_{d+1}}^{k_{n_{d+1}}},$$
 with
 $$C_{k_{1_{d+1}},\cdots,k_{n_{d+1}}} =  2^{2(n_{d+1}) \alpha + |k|} \prod_{i=1}^{n_{d+1}} \frac{\Gamma(\alpha + k_{i_{d+1}} +\frac{1}{2})  \Gamma(\alpha+1)         }{\pi^{\frac{1}{2}}\Gamma(2\alpha + k_{i_{d+1}}+1)}.$$
\end{lemma}
%\end{proposition}
\begin{proof}
On the one hand, from the product formula for the Weinstein function defined in (\ref{e61}) and the relation (\ref{eww5}), we have:
 \begin{equation}\label{ewwww1}
 e^{-\frac{\|t\|^{2}}{2}}\Psi_{\alpha,d}(ix_{1},t)\cdots\Psi_{\alpha,d}(ix_{n},t) =\sum_{m_{1},\cdots,m_{d+1} = 0}^{\infty}\gamma_{m}
  T_{x_{1}}^{\alpha,d}\circ T_{x_{2}}^{\alpha,d}\circ ...\circ T_{x_{n-1}}^{\alpha,d} H_{m}^{\alpha,d}(x_{n})t_{1}^{m_{1}}...t_{d+1}^{2
 m_{d+1}}.
\end{equation}
On the other hand, from  the relations (\ref{e5}) and (\ref{e1515}) we get:
%$$e^{-\frac{\|t\|^{2}}{2}}\Psi_{\alpha,d}(ix_{1},t)\cdots\Psi_{\alpha,d}(ix_{n},t)$$
$$e^{-\frac{\|t\|^{2}}{2}}\Psi_{\alpha,d}(ix_{1},t)\cdots\Psi_{\alpha,d}(ix_{n},t)=e^{-\frac{\|t\|^{2}}{2}} e^{<x_{1}^{\prime},t^{\prime}>}j_{\alpha}(ix_{1_{d+1}}t_{d+1}) \cdots e^{<x_{n}^{\prime},t^{\prime}>}j_{\alpha}(ix_{n_{d+1}}t_{d+1}),$$
%$$=\frac{(\Gamma(\alpha+1))^{n_{d+1}}}{\pi^{\frac{n_{d+1}}{2}}(\Gamma(\alpha+\frac{1}{2}))^{n_{d+1}}}\left(e^{<x_{1}^{\prime}+{x_{2}^{\prime}}+\cdots+{x_{n}^{\prime}}, t^{\prime}>}e^{-\frac{(\|t^{\prime}\|)^{2}}{2}}\right)\int_{[-1,1]^{n_{d+1}}}e^{-\frac{t_{d+1}^{2}}{2}+t_{d+1}(x_{1_{d+1}}u_{1_{d+1}}+\cdots + x_{n_{d+1}}u_{n_{d+1}})} w_{\alpha}(u)du$$
then
$$ e^{-\frac{\|t\|^{2}}{2}}\Psi_{\alpha,d}(ix_{1},t)\cdots\Psi_{\alpha,d}(ix_{n},t)$$
%\end{equation}$$
\begin{equation}\label{e1500}
=a_{n_{d+1}}(\alpha)\left(\prod_{k=1}^{d} e^{-\frac{t_{k}^{2}}{2}} e^{t_{k}\sum _{j=1}^{n}x_{j_{k}}}\right)\displaystyle\int_{[-1,1]^{n_{d+1}}}e^{-\frac{t_{d+1}^{2}}{2}+t_{d+1}(x_{1_{d+1}}u_{1_{d+1}}+...+ x_{n_{d+1}}u_{n_{d+1}})} w_{\alpha}(u)du,
\end{equation}
where
$$w_{\alpha}(u)= \prod_{k=1}^{n_{d+1}}(1-u_{k_{d+1}}^{2})^{\alpha-\frac{1}{2}}\quad and \quad    a_{n_{d+1}}(\alpha)=\frac{(\Gamma(\alpha+1))^{n_{d+1}}}{\pi^{\frac{n_{d+1}}{2}}(\Gamma(\alpha+\frac{1}{2}))^{n_{d+1}}} .$$
The generating function for the Hermite polynomials gives:
\begin{equation}\label{e187}
 e^{t(y_{1}u_{1}+\cdots +y_{n}u_{n})-\frac{t^{2}}{2}}= \sum_{k=0}^{\infty} H_{k}(y_{1}u_{1}+\cdots +y_{n}u_{n})\frac{t^{k}}{ k!}.
 \end{equation}
 From the identities (\ref{e1500}) and (\ref{e187}), we have:
 $$ e^{-\frac{\|t\|^{2}}{2}}\Psi_{\alpha,d}(ix_{1},t)\cdots\Psi_{\alpha,d}(ix_{n},t ) =\prod_{k=1}^{d}\left(\sum_{l_{k}=0}^{\infty}H_{l_{k}}(x_{1_{k}}+...+x_{n_{k}})\frac{t_{k}^{l_{k}}}{ l_{k}!}\right) $$ $$\times
 \left(a_{n_{d+1}}(\alpha)\int_{[-1,1]^{n_{d+1}}}\left[\sum_{l_{d+1}=0}^{\infty} H_{l_{d+1}}(x_{1_{d+1}}u_{1_{d+1}}+...+x_{n_{d+1}}u_{n_{d+1}})\frac{ t_{d+1}^{l_{d+1}}}{l_{d+1}!}\right]w_{\alpha}(u)du\right).$$
 Equating the coefficients of the power of t in the last relation and (\ref{ewwww1}), we obtain:\\
$$ \gamma_{m}
 T_{x_{1}}^{\alpha,d}\circ T_{x_{2}}^{\alpha,d}\circ ...\circ T_{x_{n-1}}^{\alpha,d} H_{m}^{\alpha,d}(x_{n}) = \left[\prod_{k=1}^{d}H_{m_{k}}(x_{1_{k}}+\cdots+x_{n_{k}}) \frac{1}{m_{k }!}\right] $$ $$ \times \left[
  a_{n_{d+1}}(\alpha)\int_{[-1,1]^{n_{d+1}}}\frac{1}{ 2m_{d+1
  }!}H_{2m_{d+1}}(x_{1_{d+1}}u_{1_{d+1}}+\cdots +x_{n_{d+1}}u_{n_{d+1}}) w_{\alpha}(u)du\right].$$
   From lemma \ref{e000090}, we get:
$$ a_{n_{d+1}}(\alpha)\int_{[-1,1]^{n_{d+1}}}\frac{1}{2m_{d+1}!}H_{2m_{d+1}}(x_{1_{d+1}}u_{1_{d+1}}+\cdots + x_{n_{d+1}}u_{n_{d+1}}) w_{\alpha}(u)du$$
\begin{eqnarray*}
&=&\sum_{s=0}^{2m_{d+1}}\frac{2^{2m_{d+1}}}{(2m_{d+1}-s)!} \tilde{h_{s}}(x_{1_{d+1}},\cdots,x_{n_{d+1}})H_{2m_{d+1}-s}(x_{1_{d+1}}+\cdots+x_{n_{d+1}}),
\end{eqnarray*}
where
\begin{eqnarray*}
\tilde{h_{s}} (x_{1_{d+1}},\cdots,x_{n_{d+1}})&=&
 \end{eqnarray*}
 $$ a_{n_{d+1}}(\alpha)\int_{[-1,1]^{n_{d+1}}} h_{s}(x_{1_{d+1}}(u_{1_{d+1}}-1),\cdots,x_{n_{d+1}}(u_{n_{d+1}}-1))w_{\alpha}(u)du,$$
 $$=(-1)^{s}\sum_{|k|=s}\frac{1}{k!}C_{k_{1_{d+1}},\cdots,k_{n_{d+1}}}x_{1_{d+1}}^{k_{1_{d+1}}}\cdots x_{n_{d+1}}^{k_{n_{d+1}}},$$

with

$$C_{k_{1_{d+1}},\cdots,k_{n_{d+1}}} =  2^{2(n_{d+1}) \alpha + |k|} \prod_{i=1}^{n_{d+1}} \frac{\Gamma(\alpha + k_{i_{d+1}} +\frac{1}{2})  \Gamma(\alpha+1)         }{\pi^{\frac{1}{2}}\Gamma(2\alpha + k_{i_{d+1}}+1)}.$$
%By a similar proof as (\cite{B.K, F.B})
We can prove  for
\begin{center}
 $ \left\{
 \begin{array}{ccc}
                                                                   x_{1_{j}}^{2} = x_{2_{j}}^{2}= \cdots = x_{n_{j}}^{2} =\frac{1}{n_{j}}, \quad 1\leq j\leq d,\\
                                                                       x_{1_{d+1}}^{2} =\cdots= x_{n_{d+1}}^{2} = \frac{4\alpha +4}{n_{d+1}}\quad or\quad 0, \\
                                                                   \end{array}\right .$ \

\end{center}
 that the Hermite polynomial $H_{l} (x_{1_{i}}+\cdots + x_{n_{i}})$ for $ 1\leq i\leq d+1$ can be approximated by the elementary symmetric functions $\sigma_{n_{i}l}(x_{1_{i}},\cdots,x_{n_{i}})$, $ 1\leq l \leq n_{i}$, where
$$\sigma_{n_{i}l}(x_{1_{i}},..,x_{n{i}}) = \sum_{1\leq j_{1} <...< j_{l} \leq n_{i}   }   x_{j_{1}}...x_{j_{l}},$$
as follows:
\begin{equation}\label{e1604}
 H_{l} (x_{1_{i}}+\cdots + x_{n_{i}}) =  l! \sigma_{n_{i}l}(x_{1_{i}},..,x_{n_{i}})  + \frac{1}{n_{i}}\sum_{r=1}^{\frac{[l]}{2}} b_{l,r}(n_{i}) H_{l-2r} (x_{1_{i}}+\cdots + x_{n_{i}}),
 \end{equation}
 where
$ b_{l,r}(n_{i})$ are bounded with respect to $n_{i}$ for a fixed $l$, (see\cite{B.K, F.B}).\\

 Thus we get
\begin{equation}\label{e1704}
\prod_{k=1}^{d}H_{m_{k}}(x_{1_{k}}+\cdots+x_{n_{k}})=
\end{equation}
$$\prod_{k=1}^{d}\left(  m_{k}! \sigma_{n_{k}m_{k}}(x_{1_{k}},\cdots,x_{n_{k}})  + \frac{1}{n_{k}}\sum_{r=1}^{[\frac{m_{k}}{2}]} b_{m_{k},r}(n_{k}) H_{m_{k}-2r} (x_{1_{k}}+\cdots + x_{n_{k}})\right),$$
and $$\sum_{s=0}^{2m_{d+1}}\frac{2^{2m_{d+1}}}{(2m_{d+1}-s)!}\tilde{h_{s}}(x_{1_{d+1}},\cdots,x_{n_{d+1}})H_{2m_{d+1}-s}(x_{1_{d+1}}+...+x_{n_{d+1}})$$
 \begin{equation}\label{e9090}
 =\phi_{m_{d+1}}(x_{1_{d+1}},...,x_{n_{d+1}}) + \frac{1}{n_{d+1}}\mathcal P_{m_{d+1}}(x_{1_{d+1}},...,x_{n_{d+1}}),
\end{equation}
where
$$\phi_{m_{d+1}}(x_{1_{d+1}},...,x_{n_{d+1}})=\sum_{s=0}^{2m_{d+1}} 2^{2m_{d+1}}\tilde{h_{s}}(x_{1_{d+1}},...,x_{n_{d+1}})\sigma_{n_{d+1} 2m_{d+1}-s}(x_{1_{d+1}},...,x_{n_{d+1}}),$$
 and$$  \mathcal P_{m_{d+1}}(x_{1_{d+1}},...,x_{n_{d+1}})=\sum_{s=0}^{2m_{d+1}}\sum_{r=1}^{[ \frac{2m_{d+1}-s}{2}]}b_{2m_{d+1}-s,r}(n_{d+1}) H_{2m_{d+1}-s-2r}(x_{1_{d+1}}+...+x_{n_{d+1}}).$$
 The result follows from  relations (\ref{e1704}) and (\ref{e9090}).
\end{proof}
\subsection{Babenko Inequality}
This subsection is devoted to establish the Babenko-Bechner-type inequality for the Weinstein transform.\\
 The product measures $\beta_{n}(x_{1})...\beta_{n}(x_{n} )$ defined in (\ref{e6092}) are discrete, all functions over these measures spaces can be identified as polynomials of the form  $\prod_{k=1}^{n}P(x_{k})$ where\\$P(x_{k})=\prod_{j=1}^{d}\left( a_{k_{j}}+b_{k_{j}}x_{k_{j}}\right)\left( a_{k_{d+1}}+ b_{1_{d+1}}x_{k_{d+1}}^{2}\right)$; $ \forall $ $ x_{k}=(x_{k_{1}},\cdots,x_{k_{d+1}}) \in \mathbb{R}_{+}^{d+1} $.\\
We define an analogue $C$ of the multiplier $K_{w}$ on the measure space over $\beta $ defined in (\ref{e1602}).%\\From the relation (\ref{e27}), we have:
 $$ K_{w}:  a H_{0}^{\alpha,d} + b H_{1}^{\alpha,d} \mapsto a H_{0}^{\alpha,d} + b  w^{d+2} H_{1}^{\alpha,d}$$
 $$ C : a + b \left[ \prod _{j=1}^{d}x_{j}\left\{(2 \alpha +2) - x_{d+1}^{2} \right\} \right] \rightarrow  a +   b w^{d+2}\left[ \prod _{j=1}^{d}x_{j}\left\{(2 \alpha +2) - x_{d+1}^{2} \right\} \right].$$
 %\quad x=(x_{1},\cdots x_{d+1})  $$
The operator $ C$ :$ L^{p}(\beta)\rightarrow L^{q}(\beta)$ is bounded, (see \cite{B.K}).\\
For $ 1\leq k \leq n $, define operators:
$$ C_{n,k} =  a + b \left[ \prod _{j=1}^{d}x_{j,k}\left\{(2 \alpha +2) - x_{d+1,k}^{2} \right\} \right] \mapsto a +  b  w^{d+2}\left[ \prod _{j=1}^{d}x_{k,j}\left\{(2 \alpha +2) - x_{k,d+1}^{2} \right\} \right]$$
where $a$ and $b$ are functions of the remaining $n-1$ vectors.\\
and defines
\begin{equation}\label{e1a3}
D_{w,n} =C_{n,1}.C_{n,2}...C_{n,n}.
\end{equation}
$D_{w,n}$ is a bounded linear operator in $ L^{p}\left[\beta_{n}(x_{1})...\beta_{n}(x_{n} )\right]$ to $ L^{q}\left[ \beta_{n}(x_{1})...\beta_{n}(x_{n} ) \right]$, (see\cite{B.K}and \cite{F.A}).\\
The restriction $\overline{D_{w,n}} $ of the operator $ D_{w,n}$ %to the subspace of function symmetric in the $n$ variables
will also be a linear operator. We denote this function space of symmetric functions over the product measures $\beta_{n}(x_{1})...\beta_{n}(x_{n} )$ by $X_{n}$.
\begin{theorem}\label{eàà}
Let $ 1 < p \leq 2 $,
 with $\frac{1}{p}+ \frac{1}{q} = 1$ and $ w = i\sqrt{p-1}$. Then the operator $ K_{w} $  satisfies the following inequality:
\begin{equation}\label{e2580}
\| K_{w} (f ) \|_{q, \mu_{\alpha,d}} \leq \| f \|_{p, \mu_{\alpha,d}}.
\end{equation}
\end{theorem}
\begin{proof}
%It suffices to prove this result for a dense set of functions, namely polynomials.\\
 Using the relation  expressed in (\ref{e28}), the polynomials $  T_{x_{1}}^{\alpha,d}\circ T_{x_{2}}^{\alpha,d}\circ ...\circ T_{x_{n-1}}^{\alpha,d} H_{m}^{\alpha,d}(x_{n})$ are approximated by %the homogenous symmetric
the polynomials $ \mathcal \phi_{m,n}^{\alpha,d}(x_{1},...x_{n})$. Since, the operator $ \overline{D}_{w,n} $ defined in (\ref{e1a3}) acts on $ \mathcal \phi_{m,n}^{\alpha,d}(x_{1},..,x_{n})$ as:
\begin{equation}\label{e385}
\overline{D}_{w,n} \mathcal \phi_{k,n}^{\alpha,d}(x_{1},..,x_{n}) =\mathcal\phi_{k,n}^{\alpha,d}( wx_{1},..,wx_{n})= w^{|k|+k_{d+1}} \mathcal \phi_{k,n}^{\alpha,d}(x_{1},..,x_{n }),
 \end{equation}
then the natural replacement for $K_{w}$ is the operator $\overline{D}_{w,n}$.\\
 Let $g(x)$ be a polynomial, then there exist a vectors $v_{1},\cdots v_{M}$ such that:
$$ g(x) =\sum_{l= 0}^{M} v_{l} H_{l}^{\alpha,d}(x); \quad  v_{l}= (v_{l_{0}},..,v_{l_{d+1}}).$$
% where $ v_{l}= (v_{l_{0}},..,v_{l_{d+1}}) $  ;\quad $ 0 < l < M $.\\
Put
 $$ G_{n}(x_{1},...,x_{n}) = \sum_{l = 0}^{M} v_{l} \phi_{l,n}^{\alpha,d}(x_{1},..,x_{n}). $$
By  lemma \ref{e0000054}, we have:
\begin{equation}\label{eA81}
 T_{x_{1}}^{\alpha,d}\circ T_{x_{2}}^{\alpha,d}\circ ...\circ T_{x_{n-1}}^{\alpha,d}(g)(x_{n}) - G_{n}(x_{1},...,x_{n}) = \sum_{l = 0}^{M} v_{l}\mathcal P_{l}^{\alpha,d}(x_{1},...,x_{n}),
\end{equation}
and
\begin{equation}\label{eA8}
T_{x_{1}}^{\alpha,d}\circ T_{x_{2}}^{\alpha,d}\circ ...\circ T_{x_{n-1}}^{\alpha,d} K_{w}(g)(x_{n}) - \overline{D}_{n}G_{n}(x_{1},...,x_{n})=\sum_{l = 0}^{M} v_{l} w^{|l|+l_{d+1}}\mathcal P_{l}^
{\alpha,d}(x_{1},...,x_{n}).
\end{equation}
Thus
$$\lim_{n_{i}\rightarrow \infty}\left(\int_{(\mathbb{R}_{+}^{d+1})^{|n|}} \left|  T_{x_{1}}^{\alpha,d}\circ T_{x_{2}}^{\alpha,d}\circ ...\circ T_{x_{n-1}}^{\alpha,d}(g)(x_{n})\right|^{p} d\beta_{n}(x_{1})...d\beta_{n}(x_{n} )\right)^\frac{1}{p}\quad 1\leq i \leq d+1$$
$$ =\lim_{n_{i}\rightarrow \infty}\left(\int_{(\mathbb{R}_{+}^{d+1})^{|n|}} \left|G_{n}(x_{1},...,x_{n})\right|^{p} d\beta_{n}(x_{1})...d\beta_{n}(x_{n} )  \right)^\frac{1}{p}, $$ and
$$\lim_{n_{i}\rightarrow \infty}\left(\int_ {(\mathbb{R}_{+}^{d+1})^{|n|}} \left|T_{x_{1}}^{\alpha,d}\circ T_{x_{2}}^{\alpha,d}\circ ...\circ T_{x_{n-1}}^{\alpha,d}( K_{w}g)(x_{n})\right|^{q} d\beta_{n}(x_{1})...d\beta_{n}(x_{n} )\right)^\frac{1}{q}$$
$$ =\lim_{n_{i}\rightarrow \infty}\left(\int_{(\mathbb{R}_{+}^{d+1})^{|n|}} \left|\overline{D}_{w,n}G_{n}(x_{1},...,x_{n})\right|^{q} d\beta_{n}(x_{1})...d\beta_{n}(x_{n} )\right)^\frac{1}{q},  $$ where $(\mathbb{R}_{+}^{d+1})^{|n|}= \prod_{k=1}^{d}\mathbb{R}^{n_{k}}\times\mathbb{R}_{+}^{n_{d+1}}.$\\

The identities (\ref{e2}) and (\ref{e162}) lead to:
$$\lim_{n_{i}\rightarrow \infty}\left(\int_{(\mathbb{R}_{+}^{d+1})^{|n|}} \left|  T_{x_{1}}^{\alpha,d}\circ T_{x_{2}}^{\alpha,d}\circ ...\circ T_{x_{n-1}}^{\alpha,d}(g)(x_{n})\right|^{p} d\beta_{n}(x_{1})...d\beta_{n}(x_{n} )\right)^\frac{1}{p} $$$$=\lim_{n_{i}\rightarrow \infty}\left(\int_{\mathbb{R}_{+}^{d+1}}\left|g (x)\right|^{p} d\gamma_{n}\right)^{\frac{1}{p}},$$then
\begin{equation}\label{e191}
\lim_{n_{i}\rightarrow \infty}\left(\int_{(\mathbb{R}_{+}^{d+1})^{|n|}} \left|  T_{x_{1}}^{\alpha,d}\circ T_{x_{2}}^{\alpha,d}\circ ...\circ T_{x_{n-1}}^{\alpha,d}(g)(x_{n})\right|^{p} d\beta_{n}(x_{1})...d\beta_{n}(x_{n} ) \right)^\frac{1}{p}
\end{equation}
$$=\left(\int_{\mathbb{R}_{+}^{d+1}}\left|g (x)\right|^{p} d\mu_{\alpha,d}(x)\right)^{\frac{1}{p}}.$$
By a similar argument
\begin{equation}\label{e4040}
\lim_{n_{i}\rightarrow \infty}\left(\int_{(\mathbb{R}_{+}^{d+1})^{|n|}}\left|T_{x_{1}}^{\alpha,d}\circ T_{x_{2}}^{\alpha,d}\circ ...\circ T_{x_{n-1}}^{\alpha,d}(K_{w}g)(x_{n})\right|^{q} d\beta_{n}(x_{1})...d\beta_{n}(x_{n} )\right)^\frac{1}{q}
\end{equation}
$$=\left(\int_{\mathbb{R}_{+}^{d+1}}\left|K_{w}g (x)\right|^{q} d\mu_{\alpha,d}(x)\right)^{\frac{1}{q}}.$$

%By the triangle inequality, we get:
%\begin{equation}\label{e5050}
%\left| \|K_{w}(g)\||_{L^{q}(d\gamma_{n})} -\|\overline{D}_{w,n}(G_{n})\|_{L^{q}(X_{n})}\right|
% \end{equation}
% $$\leq \{\int_{(\mathbb{R}_{+}^{d+1})^{|n|}} |T_{x_{1}}^{\alpha,d}\circ T_{x_{2}}^{\alpha,d}\circ ...\circ T_{x_{n-1}}^{\alpha,d} K_{w}(g)(x_{n}) - \overline{D}_{w,n}G_{n}(x_{1},...,x_{n})|^{q}d\beta(\sqrt{n}x_{1})...d\beta(\sqrt{n}x_{n})  \}^{\frac{1}{q}}.$$
  By the relations (\ref{e191})and (\ref{e4040}), we obtain:
$$\| K_{w} (g ) \|_{q, \mu_{\alpha,d}} \leq \| g \|_{p, \mu_{\alpha,d}} .$$
By density, the result is deduced
(see\cite{B.K}).
%(see\cite[p.166]{B.K.75})
\end{proof}
% In the following, we will show that from the theorem \ref{eàà}, we can find the Babenko Inequality.
\begin{theorem}( \textbf{Babenko Inequality})\\
Let $ 1 < p \leq 2$, $ q = \frac{p}{p-1}$ and let $ A_{p} = \frac{p^{\frac{1}{p}}}{q^{\frac{1}{q}}} $ . If $ f \in L^{p}( \mathbb{R}_{+}^{d+1}; d\nu_{\alpha,d})$, then the Weinstein transform $ \mathcal{F}_{w}^{\alpha,d}(f)\in L^{q}( \mathbb{R}_{+}^{d+1}; d\nu_{\alpha,d})$ and we have: %following Babenko Inequality :
\begin{equation}\label{e2180}
 \|\mathcal{F}_{w}^{\alpha,d}(f)\|_{q,\nu_{\alpha,d}} \leq A_{p}^{\alpha + \frac{d}{2} + 1 } \| f \|_{p,\nu_{\alpha,d}}.
 \end{equation}
\end{theorem}
\begin{proof}
Let $ w= i\sqrt{p-1}$, $ x = \sqrt{p}u $, $ y = \sqrt{q}v $ and $ q = \frac{p}{p-1}$, then  the kernel $k_{w}(x,y)$ given by the relation (\ref{e26}) becomes:
\begin{equation}\label{e348}
K_{w}(x,y)=\frac{1}{p^{\alpha +1 + \frac{d}{2}}} e^{\frac{(p-1)\|u\|^{2}}{2} } e^{\frac{\|v\|^{2}}{2}} \Psi_{\alpha,d} \left(u,v\right).
\end{equation}
Using the relation (\ref{e162}), we have:
\begin{equation}\label{e2654}
d\mu_{\alpha ,d} ( x ) =   e^{\frac{-p\|u\|^{2}}{2} }\frac{p^{\alpha+1 + \frac{d}{2}} u_{d+1}^{2\alpha+1}} {( 2\pi )^{     \frac{d}{2}}  2 ^{\alpha} \Gamma ( \alpha + 1) } du,
\end{equation}
and
\begin{equation}\label{e101}
d\mu_{\alpha ,d} ( y ) =   e^{\frac{-q\|v\|^{2}}{2} }\frac{q^{\alpha+1 + \frac{d}{2}} v_{d+1}^{2\alpha+1}} {( 2\pi )^{     \frac{d}{2}}  2 ^{\alpha} \Gamma ( \alpha + 1) } dv.
\end{equation}
Hence, by relations  (\ref{e42}) and (\ref{e348}), we have:

$$\int_{\mathbb{R}_{+}^{d+1}}K_{w}(x,y)g(x) d\mu_{\alpha,d}(x)$$
\begin{eqnarray*}
&=& \int_{\mathbb{R}_{+}^{d+1}}\frac{1}{p^{\alpha +1 + \frac{d}{2}}} e^{\frac{(p-1)\|u\|^{2}}{2} } e^{\frac{\|v\|^{2}}{2}} \Psi_{\alpha,d} \left(u, v \right) g(\sqrt{p}u)e^{\frac{-p\|u\|^{2}}{2} }\frac{p^{\alpha+1+\frac{d}{2}} u_{d+1}^{2\alpha+1}} {( 2\pi )^{     \frac{d}{2}}  2 ^{\alpha} \Gamma ( \alpha + 1) } du \\
&=&e^{\frac{\|v\|^{2}}{2}}\int_{\mathbb{R}_{+}^{d+1}} g(\sqrt{p}u)e^{-\frac{\|u\|^{2}}{2}}  \Psi_{\alpha,d} \left(u, v \right)d\nu_{\alpha,d}(u).
\end{eqnarray*}
If we put
\begin{equation}\label{e2800}
f(u) = g(\sqrt{p}u)e^{-\frac{\|u\|^{2}}{2}},
\end{equation}
 then, from the Weinstein transform defined in (\ref{e49}), we get
\begin{equation}\label{e11111178}
\left|\int_{\mathbb{R}_{+}^{d+1}}K_{w}(x,y)g(x) d\mu_{\alpha,d}(x)\right|^{q}=e^{\frac{q\|v\|^{2}}{2}} \left| \mathcal{F}_{w}^{\alpha,d}(f)(v)\right|^{q}.
\end{equation}
So, by the relations (\ref{e162}), (\ref{e101}) and (\ref{e11111178}), we obtain:
\begin{eqnarray*}
\|K_{w}(g)\|_{q,\mu_{\alpha,d}}   &=&\left(\int_{\mathbb{R}_{+}^{d+1}} \left|K_{w}(g)(y)\right|^{q}  d\mu_{\alpha,d}(y)\right)^{\frac{1}{q}}\\
&=& \left(\int_{\mathbb{R}_{+}^{d+1}} \left|\int_{\mathbb{R}_{+}^{d+1}}K_{w}(x,y)g(x) d\mu_{\alpha,d}(x)\right|^{q} d\mu_{\alpha,d}(y)\right)^{\frac{1}{q}}\\
&=&\left(\int_{\mathbb{R}_{+}^{d+1}}e^{\frac{q\|v\|^{2}}{2}} \left| \mathcal{F}_{w}^{\alpha,d}(f)(v)\right|^{q} d\mu_{\alpha,d}(y)\right)^{\frac{1}{q}}\\
&=&\left(\int_{\mathbb{R}_{+}^{d+1}}e^{\frac{q\|v\|^{2}}{2}} \left|\mathcal{F}_{w}^{\alpha,d}(f)(v)\right|^{q}e^{\frac{-q\|v\|^{2}}{2} }\frac{q^{\alpha+1 + \frac{d}{2}} v_{d+1}^{2\alpha+1}} {( 2\pi )^{     \frac{d}{2}}  2 ^{\alpha} \Gamma ( \alpha + 1) } dv\right)^{\frac{1}{q}}
\\&=&\left(\int_{\mathbb{R}_{+}^{d+1}} q^{\alpha+1+\frac{d}{2}}\left|\mathcal{F}_{w}^{\alpha,d}(f)(v)\right|^{q}d\nu_{\alpha,d}(v)\right)^{\frac{1}{q}}
\end{eqnarray*}
we deduce that:
 \begin{equation}\label{e20008}
\|K_{w}(g)\|_{q,\mu_{\alpha,d}} = q^{\frac{\alpha+ \frac{d}{2}+1}{q}} \| \mathcal{F}_{w}^{\alpha,d}(f)\|_{q,\nu_{\alpha,d}}.
\end{equation}
On the other hand, the identities (\ref{e2654}), (\ref{e2800}) lead to:
\begin{eqnarray*}
\|g\|_{p,\mu_{\alpha,d}}&=&\left(\int_{\mathbb{R}_{+}^{d+1}}\left|g(x)\right|^{p}d\mu_{\alpha,d}(x)\right)^{\frac{1}{p}}\\
 &=&\left(\int_{\mathbb{R}_{+}^{d+1}}\left|g(\sqrt{p}u)\right|^{p} e^{\frac{-p\|u\|^{2}}{2}}\frac{p^{\alpha+1+\frac{d}{2}}}  {( 2\pi )^{     \frac{d}{2}}  2 ^{\alpha} \Gamma ( \alpha + 1) }u_{d+1}^{2\alpha+1} du \right)^{\frac{1}{p}}\\
 &=&\left(\int_{\mathbb{R}_{+}^{d+1}} p^{\alpha+1+\frac{d}{2}}\left|f(u)\right|^{p} d\nu_{\alpha,d}(u)\right)^{\frac{1}{p}}.
\end{eqnarray*}
So, we obtain:
\begin{equation}\label{e4428}
 \|g\|_{p,\mu_{\alpha,d}} = p^{\frac{\alpha+ \frac{d}{2}+1}{p}} \|f\|_{p,v_{\alpha,d}}.
 \end{equation}
 Using the relations (\ref{e2580}), (\ref{e20008}) and (\ref{e4428}), we get:
$$\|\mathcal{F}_{w}^{\alpha,d}(f)\|_{q, \nu_{\alpha,d}} \leq A_{p}^{\frac{d}{2} +\alpha + 1} \|f\|_{p, \nu_{\alpha,d}}.$$
\end{proof}
%In the following, the new Young's-type inequality is a consequence of the Babenko-Beckner-type inequality for the weinstein transform given by the relation (\ref{e2180}).
\section {Application}
\begin{proposition}(\textbf{ Young's-type inequality})\\
Let p, q and r three real numbers such that $ 1\leq p, q, r \leq 2 $ and $\frac{1}{r}=\frac{1}{p}+\frac{1}{q}-1$. Then  for $ f\in L^{p}(\mathbb{R}_{+}^{d+1}; d\nu_{\alpha,d})$ and $ g\in L^{q}(\mathbb{R}_{+}^{d+1}; d\nu_{\alpha,d})$,  we have $f*_{w} g \in L^{r}(\mathbb{R}_{+}^{d+1}; d\nu_{\alpha,d})$ and
\begin{equation}\label{e4422222222222228}
\|f*_{w}g\|_{r,\nu_{\alpha,d}}\leq \left(\frac{ r^{\frac{1}{r}}p^{ \frac{1}{p} }q^{ \frac{1}{q}} }{ r_{1}^{\frac{1}{r}_{1}}p_{1}^{\frac{1}{p_{1}}} q_{1}^{\frac{1}{q_{1}}} }\right)^{\alpha + \frac{d}{2}+1} \|f\|_{p,\nu_{\alpha,d}}  \|g\|_{q,\nu_{\alpha,d}},
\end{equation}
where $ p_{1}= \frac{p}{p-1}$, $ q_{1}= \frac{q}{q-1}$ and $ r_{1}= \frac{r}{r-1}$.
\end{proposition}
\begin{proof}
By the identities (\ref{e159}), (\ref{e434}) and (\ref{e2180}), we have:
 \begin{eqnarray*}
\|f*_{w}g\|_{r,\nu_{\alpha,d}}&\leq& \left( \frac{r_{1}^{\frac{1}{r_{1}}}   }{r^{\frac{1}{r}}}\right)^{\alpha + \frac{d}{2}+1}\|\mathcal{F}_{w}^{\alpha,d}(f*_{w}g)\|_{r_{1,\nu_{\alpha,d}}}\\
 &\leq&\left( \frac{r_{1}^{\frac{1}{r_{1}}}   }{r^{\frac{1}{r}}}\right)^{\alpha + \frac{d}{2}+1}\|\mathcal{F}_{w}^{\alpha,d}(f) \mathcal{F}_{w}^{\alpha,d}(g)\|_{r_{1,\nu_{\alpha,d}}}\\
&\leq&\left( \frac{r_{1}^{\frac{1}{r_{1}}}   }{r^{\frac{1}{r}}}\right)^{\alpha + \frac{d}{2}+1}\|\mathcal{F}_{w}^{\alpha,d}(f)\|_{p_{1,\nu_{\alpha,d}}}\|\mathcal{F}_{w}^{\alpha,d}(g)\|_{q_{1,\nu_{\alpha,d}}}\\
&\leq&\left(\frac{ r^{\frac{1}{r}}p^{ \frac{1}{p} }q^{ \frac{1}{q}} }{ r_{1}^{\frac{1}{r}_{1}}p_{1}^{\frac{1}{p_{1}}} q_{1}^{\frac{1}{q_{1}}} }\right)^{\alpha + \frac{d}{2}+1} \|f\|_{p,\nu_{\alpha,d}}  \|g\|_{q,\nu_{\alpha,d}}.
\end{eqnarray*}
\end{proof}

\end{document}